\documentclass[a4paper,11pt]{article}
\usepackage{bbding,textcomp,amsfonts,amsthm,amsmath,mathrsfs}
\usepackage{amssymb}
\usepackage{multirow, url}
\usepackage[utf8x]{inputenc}
\usepackage{comment}
\usepackage[affil-it]{authblk}
\usepackage{titling}
\usepackage[T1]{fontenc}
\usepackage{ae,aecompl}
\usepackage{enumerate}

\newtheorem{theorem}{Theorem}

\numberwithin{subcase}{case}

\PrerenderUnicode{\unichar{355}}

\title{On spanning trees with high internal degree}

\author{Codru\unichar{355} Grosu\thanks{This research was supported by the Deutsche Forschungsgemeinschaft within the research training group `Methods for Discrete Structures' (GRK 1408).}}
\affil{\small{grosu.codrut@gmail.com}}
\date{}

\begin{document}

\maketitle

\abstract{Alon and Wormald showed that any graph with minimum degree $d$ contains a spanning star forest in which every connected component is of size at least $\Omega((d/\log d)^{1/3})$. They asked if any connected graph with minimum degree at least $d$ has a spanning tree in which every internal vertex has degree at least $cd/\log d$, for some absolute constant $c > 0$.

We give a simple example showing that this is not the case.
}

\section{\normalsize The example}

Let $G$ be any graph. A \textit{star factor} of $G$ is a spanning forest of $G$ in which every connected component is a star. Answering a question of Havet, Klazar, Kratochvil, Kratsch and Liedloff (\cite{Havet11}), Alon and Wormald (\cite{AlonWormald10}) proved the following result.

\begin{theorem}[\cite{AlonWormald10}]
\label{thm:alon}
There exists a $c > 0$ such that for all $d \geq 2$, every graph
with minimum degree $d$ contains a star factor in which every star has at least $c\left(\frac{d}{\log d}\right)^{1/3}$ edges.
\end{theorem}
As noted by Alon and Wormald, any $n$-vertex connected graph of minimum degree $d$ has a spanning tree with at least $n - O(n\frac{\log d}{d})$ leaves (see \cite{West91}). In view of this and Theorem \ref{thm:alon}, they asked if there exists a constant $c > 0$, such that any connected graph of minimum degree at least $d$ has a spanning tree in which every internal vertex is of degree at least $c\frac{d}{\log d}$.

We now describe an example showing that this need not hold.

\begin{theorem}
\label{thm:main}
For any $d \geq 2$ and $n \geq d(d+2)$, there exists a connected graph on $n$ vertices with minimum degree $d$, for which every spanning tree has an internal vertex of degree $2$.
\end{theorem}
\begin{proof}
Take a complete graph $K$ on $d$ vertices and fix some vertex $x \in K$. For every vertex $y \in K \setminus \{x\}$, add a complete graph $K_y$ on $d+1$ vertices, and join $y$ to some vertex $z_y \in K_y$. Finally, add a complete graph $W$ on $n - (d-1)(d+1) - d$ vertices, and join $x$ to some vertex $z_x \in W$. Let $G$ be the resulting graph.

As $|W| \geq d+1$, any vertex in $G$ has degree at least $d$, and every vertex in $K$ has degree exactly $d$.

Let $T$ be any spanning tree of $G$. Then $T$ must contain all the edges $\{(u, z_u) : u \in K\}$, and furthermore all the vertices in $K$ are internal vertices of $T$. Moreover, $T$ induces a tree $T[K]$ on $K$. Let $\ell$ be any leaf of $T[K]$. Then $\ell$ is an internal vertex of $T$ of degree $2$. This completes the proof.
\end{proof}

The graph $G$ in Theorem \ref{thm:main} can also be modified to be regular of degree $d$ (albeit for a larger lower bound on $n$).

\bibliographystyle{abbrv}
\bibliography{bibl}
\end{document}